\documentclass[preprint,12pt]{elsarticle}
\journal{Linear Algebra and its Applications}

\usepackage[dvipsnames]{xcolor}
\usepackage[colorlinks]{hyperref}

\usepackage{amsfonts,amsmath,amssymb,amsthm,epsfig,fancyhdr,graphics}
\usepackage[inline]{enumitem}
\usepackage{cleveref}
\usepackage{tikz}
\usepackage{comment}
\usepackage{booktabs}
\usepackage{microtype}
\usepackage[linesnumbered,commentsnumbered,ruled,vlined]{algorithm2e}
\usepackage{nicefrac}

\theoremstyle{plain}
\newtheorem{theorem}{Theorem}[section]
\newtheorem{corollary}[theorem]{Corollary}
\newtheorem{proposition}[theorem]{Proposition}
\newtheorem{lemma}[theorem]{Lemma}
\theoremstyle{definition}

\newtheorem{remark}[theorem]{Remark}
\newtheorem{example}[theorem]{Example}

\DeclareMathOperator{\bs}{\mathbf{s}}
\DeclareMathOperator{\Tr}{Tr}
\DeclareMathOperator{\one}{\mathbf{1}}
\let\oldpi\pi
\renewcommand{\pi}{\ensuremath{\boldsymbol{\oldpi}}}
\let\oldalpha\alpha
\renewcommand{\alpha}{\ensuremath{\boldsymbol{\oldalpha}}}
\usepackage{relsize} %
\newcommand{\hutchpp}{\text{Hutch\raisebox{0.35ex}{\relscale{0.75}++}}}

\begin{document}

\begin{frontmatter}
\title{On Kemeny's constant and stochastic complement\tnoteref{titlelabel}}
\tnotetext[titlelabel]{This article was partially funded by the ``INdAM – GNCS Project: Metodi basati su matrici e tensori strutturati per problemi di algebra lineare di grandi dimensioni'' code CUP\_E53C22001930001, and by the PRIN ``Low-rank Structures and Numerical Methods in Matrix and Tensor Computations and their Application'' code 20227PCCKZ. S. Kim is supported in part by funding from the Fields Institute for Research in Mathematical Sciences and from the Natural Sciences and Engineering Research Council of Canada. The second and fourth authors are member of the INdAM GNCS group.}
\author[pisa]{Dario Andrea Bini}
\ead{dario.bini@unipi.it}
\author[pisa]{Fabio Durastante\corref{cor1}}
\ead{fabio.durastante@unipi.it}
\author[york]{Sooyeong Kim}
\ead{kimswim@yorku.ca}
\author[pisa]{Beatrice Meini}
\ead{beatrice.meini@unipi.it}
\cortext[cor1]{Corresponding author}
\affiliation[pisa]{organization={Mathematics Department, University of Pisa},%
            addressline={Largo Bruno Pontecorvo, 5}, 
            city={Pisa},
            postcode={56127}, 
            state={PI},
            country={Italy}}
\affiliation[york]{organization={Department of Mathematics and Statistics, York University},%
            addressline={4700 Keele Street}, 
            city={Toronto},
            postcode={M3J 1P3}, 
            state={ON},
            country={CA}}

\begin{abstract}
Given a stochastic matrix $P$ partitioned in four blocks $P_{ij}$, $i,j=1,2$, Kemeny's constant $\kappa(P)$ is expressed in terms of Kemeny's constants of the stochastic complements $P_1=P_{11}+P_{12}(I-P_{22})^{-1}P_{21}$, and $P_2=P_{22}+P_{21}(I-P_{11})^{-1}P_{12}$. Specific cases concerning periodic Markov chains and Kronecker products of stochastic matrices are investigated. %
Bounds to Kemeny's constant of perturbed matrices are given.  Relying on these theoretical results, a divide-and-conquer algorithm for the efficient computation of Kemeny's constant of graphs is designed. Numerical experiments performed on real world problems show the high efficiency and reliability of this algorithm.
\end{abstract}

\begin{highlights}
\item Expression of Kemeny's constant employing the constants of stochastic complements.
\item New recursive algorithms for computing Kemeny's constant 
\end{highlights}


\end{frontmatter}

\section{Introduction}

Given an ergodic, finite, discrete-time, time-homogeneous Markov chain, Kemeny's constant is the expected time for the Markov chain to travel between randomly chosen states, where these states are sampled according to the stationary distribution. Originally defined in \cite{ks} as the expected time to reach a randomly-chosen state from a fixed starting state, this quantity is independent of the choice of initial state. Various intuitive explanations for the constancy have been provided in \cite{bini2018kemeny,doyle2009kemeny,kirkland2021directed}. Moreover, Kemeny's constant has many applications to a variety of subjects, including the study of road traffic networks \cite{altafini,crisostomi2011google}, disease spread \cite{kim2023effectiveness,Emma:Kemeny}, and many others.

In addition, Kemeny's constant has recently received significant attention within the graph theory community. Such constant finds special relevance in the study of random walks on graphs. In the context of random walks, Kemeny's constant is used as a measure of connectivity of the graph: the smaller the constant, the faster a random walker moves around the graph.  As a graph invariant, much research has been dedicated to understanding how the structure of a graph influences Kemeny's constant. Extremal Kemeny's constant has been studied in \cite{breen19,ciardo,kim23}. The impact of edge addition/removal on this quantity has also been explored in \cite{altafini,CiardoparadoxicalTwins,faught20221,kim2022families,kirkland2023edge,KirklandZeng}. Recently, work in \cite{kim2023bounds} provides insights into the interplay between a graph and its complement regarding Kemeny's constant.

The pursuit of reducing the computational complexity of Kemeny's constant of a Markov chain with $n$ states, which is $O(n^3)$, is interesting. Randomized approaches for the direct approximation of Kemeny's constant have been used in \cite{liMC,Xukem}. Additionally, exploring computational questions where partial information is available without computing from scratch has been investigated. Work in \cite{altafini} investigated computation of Kemeny's constant for a graph obtained by removing an edge. Moreover, work from \cite{breen23} provided an explicit formula for Kemeny's constant for graphs with bridges in terms of some quantities on the subgraphs resulting from the deletion of the bridges.

%

%

In this article, we furnish a formula for Kemeny's constant of a Markov chain by utilizing two Markov chains induced from the original, known as \textit{censored (watched) Markov chains} (see \cite{meyer}). Let $P$ be the transition matrix of an irreducible Markov chain with state space $\{1,\ldots,n\}$. Let \begin{equation}\label{eq:P}
	P = \begin{bmatrix}
	P_{11} & P_{12} \\
	P_{21} & P_{22}
	\end{bmatrix},
\end{equation}
where $P_{11}$ and $P_{22}$ are square of size $m\times m$ and $(n-m)\times (n-m)$, respectively. The transition matrices $P_1$ and $P_2$ of the censored Markov chains are given by
\begin{equation}\label{eq:P1}
    \begin{split}
        P_1 = P_{11}+P_{12}(I-P_{22})^{-1}P_{21},\quad \quad P_2 = P_{22}+P_{21}(I-P_{11})^{-1}P_{12}.
    \end{split}
\end{equation}
In particular, we will provide expressions of the difference $\gamma=\kappa(P)-\kappa(P_1)-\kappa(P_2)$, given in terms  of the stationary distribution vector of the Markov chain $P$.

The paper is organized as follows: In Section~\ref{sec:not}, we introduce some notation and review the definitions and properties of stochastic complements and Kemeny's constant. Section~\ref{sec:stochastic-complements} concerns the main theoretical result (Theorem~\ref{th:kappap}), where Kemeny's constant of a Markov chain is related to Kemeny's constants of censored Markov chains. Section \ref{sec:spec}  explores applications of the main result for various structured transition matrices and subsequently determines the minimum Kemeny's constant of a periodic Markov chain. In Section \ref{sec:bounds}, we establish lower and upper bounds for the constant $\gamma$ and bounds on Kemeny's constant of perturbed matrices are given. In Section \ref{sec:algo}, we introduce a divide-and-conquer algorithm to compute $\kappa(P)$ where $P$ is a large and highly sparse matrix, and we present numerical experiments performed with real world matrices to show the effectiveness and reliability of this algorithm.  Section \ref{sec:conc} draws the conclusions.

\section{Notation and preliminaries}\label{sec:not}
A matrix $A=[a_{ij}]\in\mathbb{R}^{m\times n}$ is said to be \textit{nonnegative (resp. positive)}, if $a_{ij}\ge 0$ ($a_{ij}>0$) for any $i,j$, and we write $A\ge 0$ (resp. $A>0$). We denote by $\one_n$ the all ones vector of size $n$. If $n$ is clear from the context, we shall omit the subscript of $\one_n$. A matrix $A$ is said to be \textit{reducible} if there exists a permutation matrix $\Pi$ such that $\Pi A \Pi^T$ is a block upper triangular matrix with square diagonal blocks. If $A$ is not reducible, then we say that $A$ is \textit{irreducible}. A nonnegative matrix $A$ is said to be \textit{stochastic} if $A\one=\one$. For $A\in\mathbb{R}^{n\times n}$, we use $\rho(A)$ to denote the spectral radius of $A$.

A matrix $A$ is a \emph{non-singular (resp. singular)} M-matrix if $A = s I - B$ with $B \geq 0$ and $s > \rho(B)$ (resp. $s=\rho(B)$). It is known that if $A$ is a non-singular M-matrix, then $a_{ii} > 0$, $a_{ij} \leq 0$ for $i \neq j$, and $A^{-1} \ge  0$.

Given a vector $\mathbf{v}=(v_i)\in\mathbb{R}^n$, we denote by $\| \mathbf{v}\|$ the 1-norm of $\mathbf{v}$, \textit{i.e.}, $\| \mathbf{v}\|=\sum_{i=1}^n |v_i|$, and we use $\|\mathbf{v}\|_\infty$ to indicate the infinity norm of $\mathbf{v}$, that is, $\|\mathbf{v}\|_\infty=\max_i|v_i|$. For a matrix $A$, we denote by $\| A \|_\infty$ the norm induced by the infinity norm, that is, $\| A \|_\infty=\max_j \sum_{i=1}^n |a_{ij}|$.

\subsection{Stochastic complement}

Let $P$ be an $n\times n$ irreducible stochastic matrix, with \textit{stationary distribution vector} $\pi$, \textit{i.e.},  $\pi^TP=\pi^T$ and $\pi^T \one=1$. Let $P$ be partitioned into the $2\times 2$ block matrix as in \eqref{eq:P}.
Since $P$ is irreducible, $I-P_{11}$ and $I-P_{22}$ are non-singular M-matrices. It is found in \cite{meyer} that the matrices $P_1$ and $P_2$ in \eqref{eq:P1} are stochastic and irreducible. We call $P_1$ (resp. $P_2$) the \emph{stochastic complement} of $P_{11}$ (resp. $P_{22}$) in $P$. The matrix $P_1$ represents the transition matrix of the Markov chain obtained by censoring the states $\{m+1,\ldots,n\}$, while $P_2$ represents the transition matrix of the Markov chain obtained by censoring the states $\{1,\ldots,m\}$. This suggests the name ``censored Markov chain''.

Let $\pi$ be  partitioned conformally with $P$ so that $\pi^T = \begin{bmatrix}
\pi_1^T & \pi_2^T
\end{bmatrix}$ where $\pi_1\in\mathbb R^m$ and $\pi_2\in\mathbb R^{n-m}$. Denoting by $\hat{\pi}_i$ the stationary distribution vector of $P_i$, we have
\[
\hat \pi_i=\frac 1{\|\pi_i \|}\pi_i,~~i=1,2. 
\]
Moreover, setting $\oldalpha_i=\| \pi_i \|$, we see that
\begin{equation}\label{eq:agg1}
\pi^T=[\oldalpha_1\hat\pi_1^T,\oldalpha_2\hat\pi_2^T], %
\end{equation}
and the vector $\alpha^T=[\oldalpha_1,\oldalpha_2]$ satisfies 
\begin{equation}\label{eq:agg2}
\alpha^T S=\alpha^T,~~\alpha^T \one=1,\quad
S=\begin{bmatrix}
\hat\pi_1^T P_{11}\one_m &\hat\pi_1^T P_{12}\one_{n-m}\\
\hat\pi_2^T P_{21}\one_m&\hat\pi_2^T P_{22}\one_{n-m}
\end{bmatrix}.
\end{equation}
That is, $\alpha$ is the stationary distribution vector of the aggregated matrix $S$, which is stochastic and irreducible.

\subsection{Kemeny's constant} 

Let $P$ be the transition matrix of an ergodic, finite, discrete-time, time-homogeneous Markov chain. %
\emph{Kemeny's constant} of $P$, denoted by $\kappa(P)$, is the expected number of time steps required by the Markov chain to go from a given state $i$ to a random state $j$, sampled according to the stationary distribution $\pi$. 
Kemeny's constant has an algebraic characterization, in terms of the trace of a suitable matrix \cite{kemeny,wang2017kemeny}, as stated in the following.

\begin{lemma}
\label{lem:kem}
Let $\mathbf{g},\mathbf{h}\in\mathbb{R}^{n}$ be vectors with $\mathbf{h}^T\mathbf{g}=1$, $\mathbf{h}^T\one \neq 0$, $\pi^T \mathbf{g} \neq 0$. Then, $I-P+\mathbf{g}\mathbf{h}^T$ is non-singular, and
\begin{equation} \label{K1}
    \kappa(P) = \Tr(Z) - \pi^T Z \one.
\end{equation}
where $Z = (I-P+\mathbf{g}\mathbf{h}^T)^{-1}$.
\end{lemma}

By choosing $\mathbf{g}=\one$, we find that
\begin{equation}\label{eq:ks}
\kappa(P)=\Tr(Z)-1\quad \text{where}\quad Z = (I-P+\one \mathbf{h}^T)^{-1}.
\end{equation}
Let $\lambda_1 = 1,\lambda_2,\ldots,\lambda_n$ be the eigenvalues of $P$. Since $P$ is irreducible, $\lambda_1$ is simple. By the Brauer theorem \cite{brauer}, the eigenvalues $1-\lambda_i$ of the matrix $I-P$ coincide with the eigenvalues of $I-P+\one \mathbf{h}^T$ except for the eigenvalue $1-\lambda_1=0$ which is mapped into  1. This fact enable us to express Kemeny's constant in terms of the eigenvalues $\lambda_i$ of $P$ as
\begin{equation}\label{eq:kemeig}
    \kappa(P)=\sum_{i=2}^n\frac 1{1-\lambda_i}.
\end{equation}

Recall that given two square stochastic matrices $A,B$ of the same size, the products $AB$ and $BA$ are stochastic as well and have the same characteristic polynomial, see \cite[Theorem 1.3.20]{hj:book}. So, their eigenvalues coincide and have the same algebraic multiplicities. Thus, it follows from \eqref{eq:kemeig} that 
\begin{equation}\label{eq:kemprod}
\kappa(AB)=\kappa(BA).
\end{equation}

A slightly different situation is encountered for non-square matrices. Suppose that $A$ and $B$ are stochastic matrices of size $m\times n$ and $n\times m$, respectively, where $n>m$. Then $AB$ and $BA$ are square stochastic matrices of size $m\times m$ and $n\times n$, respectively. Moreover, their nonzero eigenvalues coincide, thus we deduce from \eqref{eq:kemeig} that
\begin{equation}\label{eq:kemdiff}
    \kappa(BA)=\kappa(AB)+n-m.
\end{equation}

\subsection{Kemeny's constant and graphs}
A graph is a pair $ G = (V, E)$, where $V$ is a finite set of \emph{vertices} of cardinality $|V|=n$, and $E$ is a set of pairs $\{i,j\}$ where $i,j\in V$, called \emph{edges}. 
A graph is \emph{weighted} if it is equipped with a function $a: E \to \mathbb{R}_+$ that associates a non-negative \emph{weight} to each edge. 
For unweighted graphs, the weight of each edge is assumed to be equal to 1.
The \emph{adjacency matrix} of a graph on $n$ vertices is the $n\times n$  matrix $A = [a_{ij}]$ such that $a_{ij}$ is equal to the weight of edge $\{i,j\}$ if $\{i,j\}\in E$, and zero otherwise.

Given a graph $G$, we can associate the Markov chain having transition matrix $P = D^{-1}A$, where $D = \textrm{diag}(\mathbf{d})$, $\mathbf{d}=A\one$, and $A$ is the adjacency matrix of the graph. Here, we assume that $d_i\ne 0$ for any $i$, that is there are no isolated vertices. The matrix $P$ describes a random walk on the graph $ G$.
We can then define Kemeny's constant of the graph as $\kappa(G) := \kappa(P)$.

A graph is said to be \emph{bipartite} if there exist disjoint sets of vertices $V_1$ and $V_2$, of cardinality $m_1$ and $m_2$, respectively,  such that $V=V_1\cup V_2$, and  the vertices in $V_1$ as well as the vertices in $V_2$ are connected by no edges. 
The adjacency matrix $H$ of a bipartite graph can be partitioned into 4 blocks $H_{i,j}$, $i,j=1,2$ where $H_{11}$ and $H_{22}$ are the null square matrices of size $m_1$, and $m_2$, respectively. If the graph is unweighted complete bipartite then the blocks $H_{12}$ and $H_{21}$ of size $m_1\times m_2$, $m_2\times m_1$, respectively,  have all the entries equal to 1.

\section{Kemeny's constant and stochastic complement}\label{sec:stochastic-complements}
In this section we relate Kemeny's constant of the transition matrix $P$ in \eqref{eq:P} to Kemeny's constants of the stochastic complements \eqref{eq:P1}. We begin with introducing the following preliminary result.

\begin{lemma}\label{lem:2}
Let $P$ be the transition matrix given by \eqref{eq:P}, and let $Z=(I-P+\one \mathbf{h}^T)^{-1}$ where $\mathbf{h}\in\mathbb{R}^{n}$ with $\mathbf{h}^T \one=1$. Partition $\mathbf{h}$ conformally with $P$ as $\mathbf{h}^T=\begin{bmatrix}
\mathbf{h}_1^T & \mathbf{h}_2^T
\end{bmatrix}$.  Then
\begin{equation}\label{eq:6}
\Tr(Z)=\Tr(Z_1)+\Tr(Z_2)    
\end{equation}
where
\begin{equation}\label{eq:q11}
\begin{split}
Z_1=(I-P_1+\sigma_2^{-1}\mathbf{u}_1 \mathbf{v}_1^T)^{-1},\quad\quad Z_2=(I-P_2+\sigma_1^{-1}\mathbf{u}_2 \mathbf{v}_2^T)^{-1},
\end{split}
\end{equation}
and
\[
\begin{split}
    &\mathbf{u}_1=\one_m+P_{12}(I-P_{22})^{-1}\one_{n-m},~~\mathbf{u}_2=\one_{n-m}+P_{21}(I-P_{11})^{-1}\one_m,\\
    &\mathbf{v}_1^T=\mathbf{h}_1^T+\mathbf{h}_2^T(I-P_{22})^{-1}P_{21},~~\mathbf{v}_2^T=\mathbf{h}_2^T+\mathbf{h}_1^T(I-P_{11})^{-1}P_{12},\\
    &\sigma_i=1+\mathbf{h}_i^T(I-P_{ii})^{-1}\one \quad \text{for}\quad i=1,2.
\end{split}
\]
\end{lemma}

\begin{proof}
By partitioning $Z$ conformally with $P$ in \eqref{eq:P}, we find that
\[
Z=	\begin{bmatrix}
	I-P_{11}+\one_m \mathbf{h}_1^T & -P_{12}+\one_m \mathbf{h}_2^T \\
	-P_{21}+\one_{n-m} \mathbf{h}_1^T & I-P_{22}+\one_{n-m} \mathbf{h}_2^T
	\end{bmatrix}^{-1}.
\]
We shall omit the subscripts of all ones vectors. Then 
\[
Z=	\begin{bmatrix}
	Z_{1} & * \\ * & Z_{2}
	\end{bmatrix},
\]
where %
\begin{equation}\label{eq:q12}
\begin{split}
&Z_1=\left(I-P_{11} +\one \mathbf{h}_1^T-(P_{12}-\one \mathbf{h}_2^T)(I-P_{22}+\one \mathbf{h}_2^T)^{-1} (P_{21}-\one \mathbf{h}_1^T)\right)^{-1},\\
&Z_2=\left(I-P_{22} +\one \mathbf{h}_2^T-(P_{21}-\one \mathbf{h}_1^T)(I-P_{11}+\one \mathbf{h}_1^T)^{-1} (P_{12}-\one \mathbf{h}_2^T)\right)^{-1},
\end{split}
\end{equation}
and $*$ denotes a generic entry. In particular,
$\Tr(Z)=\Tr(Z_1)+\Tr(Z_2)$.
By using the Sherman--Morrison formula, we find that
\[
(I-P_{22}+\one \mathbf{h}_2^T)^{-1}=(I-P_{22})^{-1}-\sigma_2^{-1}(I-P_{22})^{-1}\one \mathbf{h}_2^T(I-P_{22})^{-1},
\]
where $\sigma_2=1+\mathbf{h}_2^T(I-P_{22})^{-1}\one$. Let $\mathbf{r}_1 = (I-P_{22})^{-1}\mathbf{1}$ and $\mathbf{r}_2^T = \mathbf{h}_2^T(I-P_{22})^{-1}$. Replacing this expression in the first formula in \eqref{eq:q12} yields 
\begin{align*}
    Z_1^{-1} =&\; I-P_1+\mathbf{1}\left(\left(1-\mathbf{r}_2^T\mathbf{1}+\sigma_2^{-1}(\mathbf{r}_2^T\mathbf{1})(\mathbf{r}_2^T\mathbf{1})\right)\mathbf{h}_1^T+(1-\sigma_2^{-1}\mathbf{r}_2^T\mathbf{1})\mathbf{r}_2^TP_{21}\right)\\
    &+\;P_{12}\mathbf{r}_1\left(\sigma_2^{-1}\mathbf{r}_2^TP_{21}+(1-\sigma_2^{-1} \mathbf{r}_2^T\mathbf{1})\mathbf{h}_1^T\right).
\end{align*}
Since $\sigma_2 = 1+\mathbf{r}_2^T\mathbf{1}$, we obtain the formula for $Z_1$ in \eqref{eq:q11}. We proceed similarly for $Z_2$.
\end{proof}

Observe that the vectors $\mathbf{v}_1$ and $\mathbf{v}_2$ in Lemma~\ref{lem:2} depend on the vector $\mathbf{h}$ that can be chosen arbitrarily under the condition $\mathbf{h}^T \one=1$. In the next proposition we show that $\mathbf{h}$ can be chosen with $\mathbf{v}_1=\hat\pi_1$ and $\mathbf{v}_2=\hat \pi_2$.
 This choice will allow us to express Kemeny's constant of $P$ in terms of Kemeny's constant of the stochastic complements $P_1$ and $P_2$, which will be shown in Theorem~\ref{th:kappap}.

\begin{proposition}\label{prop:appunti}
    There exist vectors $\mathbf{h}_1\in\mathbb{R}^m$ and $\mathbf{h}_2\in\mathbb{R}^{n-m}$ such that
    \begin{equation}\label{eq:hi}
    \begin{split}
&    \mathbf{h}_1^T+\mathbf{h}_2^T(I-P_{22})^{-1}P_{21}=\hat \pi_1^T,\\
 &   \mathbf{h}_2^T+\mathbf{h}_1^T(I-P_{11})^{-1}P_{12}=\hat \pi_2^T,
 \end{split}
    \end{equation}
    and $\mathbf{h}_1^T \one_m+\mathbf{h}_2^T\one_{n-m} =1$.
\end{proposition}
\begin{proof}
Observe that \eqref{eq:hi} is equivalent to
\begin{equation}\label{eq:h2}
\begin{split}
& \mathbf{h}_1^T\left( I-(I-P_{11})^{-1}P_{12}(I-P_{22})^{-1}P_{21} \right)=\hat \pi_1^T -\hat\pi_2^T (I-P_{22})^{-1}P_{21},\\
& \mathbf{h}_2^T= \hat \pi_2^T - \mathbf{h}_1^T(I-P_{11})^{-1}P_{12}.
\end{split}
\end{equation}
Since the matrix $T=(I-P_{11})^{-1}P_{12}(I-P_{22})^{-1}P_{21}$ is stochastic, the vector $\one_m$ is orthogonal to the rows of $I-T$. On the other hand, 
since $(I-P_{22})^{-1}P_{21}$ is stochastic and $\hat \pi^T_1 \one_m=\hat\pi_2^T \one_{n-m}=1$, the vector $\one_m$ is orthogonal to the right-hand side of the first equation in \eqref{eq:h2}. Hence, the system has a solution by the Rouch\'{e}--Capelli Theorem \cite[Section~0.2.4]{hj:book}. If $\mathbf{h}_1$ is a solution, then $\mathbf{h}_2$ can be recovered from the second equation in \eqref{eq:h2}.
By multiplying to the right by $\one_m$  the first equation in \eqref{eq:hi}, we find that $\mathbf{h}_1^T \one_m + \mathbf{h}_2^T \one_{n-m}=1$.
\end{proof}

Proposition~\ref{prop:appunti}, together with Equation \eqref{eq:6} and Lemma \ref{lem:kem}, implies the following result.

\begin{theorem}\label{th:kappap}
Let $P$ be a stochastic irreducible matrix, partitioned as in \eqref{eq:P}, and denote by $P_1$ and $P_2$ the stochastic complements \eqref{eq:P1}. Then
\begin{equation}\label{formula0}
  \kappa(P) = \kappa(P_1)+\kappa(P_2)+\gamma, 
\end{equation}
where
{
\begin{equation}\label{eq:gamma}
\gamma=
\begin{bmatrix}
\widehat \pi_1^T & -\widehat \pi_2^T 
\end{bmatrix}
(I-P+\mathbf{u} \mathbf{v}^T)^{-1}
\begin{bmatrix}
\| \pi_2\| \one \\
-\| \pi_1 \| \one
\end{bmatrix},
\end{equation}
and $\mathbf{u}$, $\mathbf{v}\in\mathbb{C}^n$ are any vectors such that
$\mathbf{v}^T\one\ne 0$ and $\pi^T \mathbf{u}\ne 0$.}
\end{theorem}

\begin{proof}
From \eqref{eq:ks}, $\kappa(P)=\Tr(Z)-1$, where $Z=(I-P+\one \mathbf{h}^T)^{-1}$ and $\mathbf{h}$ is a vector such that $\mathbf{h}^T \one=1$. Partition $\mathbf{h}$ as $\mathbf{h}^T=[\mathbf{h}_1^T,\mathbf{h}_2^T]$, where $\mathbf{h}_1$ and $\mathbf{h}_2$ satisfy \eqref{eq:hi}. We shall maintain the same notation in Lemma~\ref{lem:2}. Then $\mathbf{v}_1^T=\hat\pi_1^T$ and $\mathbf{v}_2^T=\hat\pi_2^T$.
Since $\hat\pi_1^T(I-P_1)=0$, the eigenvalues of $I-P_1+\sigma_2^{-1}\mathbf{u}_1\hat \pi_1^T$ are the eigenvalues of $I-P_1$, except for the eigenvalue equal to zero, which is replaced by $\sigma_2^{-1}\hat\pi_1^T \mathbf{u}_1$. Similarly, the eigenvalues of $I-P_2+\sigma_1^{-1}\mathbf{u}_2\hat \pi_2^T$ are the eigenvalues of $I-P_2$, except for the eigenvalue equal to zero, which is replaced by $\sigma_1^{-1}\hat\pi_2^T \mathbf{u}_2$. 
In view of \eqref{eq:6}, $\Tr(Z)=\Tr(Z_1)+\Tr(Z_2)$. 
From the expression of $\mathbf{u}_1$ in Lemma \ref{lem:2}, we find that 
\[
\hat\pi_1^T \mathbf{u}_1=1+\hat\pi_1^T P_{12}(I-P_{22})^{-1}\one_{n-m}=
1+\frac{\pi_2^T \one_{n-m}}{\| \pi_1\| }=1+\frac{\|\pi_2\|}{\| \pi_1\|}=\frac{1}{\| \pi_1 \|}.
\]
Hence, $\Tr(Z_1)=\kappa(P_1)+\sigma_2{\| \pi_1\|}$. Similarly, $\hat\pi_2^T \mathbf{u}_2=\frac{1}{\| \pi_2 \|}$ and so $\Tr(Z_2)=\kappa(P_2)+\sigma_1{\| \pi_2\| }$. Since
$\Tr(Z)=\kappa(P_1)+\kappa(P_2)+{\sigma_2}{\| \pi_1\| }+{\sigma_1}{\|\pi_2 \|}$, we have $\kappa(P)=\kappa(P_1)+\kappa(P_2)+\gamma$ where 
$\gamma={\sigma_2}{\| \pi_1\| }+{\sigma_1}{\|\pi_2 \|}-1$.

Set $\mathbf{k}_i^T=\mathbf{h}_i^T(I-P_{ii})^{-1}$, $i=1,2$. Then 
\[
\gamma=\| \pi_1\| \mathbf{k}_2^T\one + \| \pi_2\| \mathbf{k}_1^T\one.
\]
We find from \eqref{eq:hi} that $\mathbf{k}_1$ and $\mathbf{k}_2$ solve the linear system
\begin{equation}\label{eq:systk1k2}
\begin{bmatrix}
\mathbf{k}_1^T & -\mathbf{k}_2^T
\end{bmatrix}
\begin{bmatrix}
I-P_{11} & -P_{12}\\
-P_{21} & I-P_{22}
\end{bmatrix}=
\begin{bmatrix}
\widehat \pi_1^T & -\widehat \pi_2^T 
\end{bmatrix}.
\end{equation}
Since $(I-P)\one=0$ and $\pi^T(I-P)=0$, this system is consistent.
 {According to \cite[Theorem 2]{kemeny}, a solution of such a system is
    \[
\begin{bmatrix}
    \mathbf{k}_1^T & -\mathbf{k}_2^T
\end{bmatrix}=
\begin{bmatrix}
\widehat \pi_1^T & -\widehat \pi_2^T 
\end{bmatrix}\widehat Z
    \]
    where
$\widehat Z=(I - P + \mathbf{u} \mathbf{v}^T)^{-1}$, 
and $\mathbf{u}$, $\mathbf{v}\in\mathbb{C}^n$ are such that $\mathbf{v}^T\one\ne 0$ and $\pi^T \mathbf{u}\ne 0$.
Therefore we arrive at \eqref{eq:gamma}.
}
\end{proof}

{
By choosing $\mathbf{u}=\one$ and $\mathbf{v}=\pi$, we have}
\[\begin{split}
(I-P+\one \pi^T)^{-1}
\begin{bmatrix}
\| \pi_2\| \one \\
-\| \pi_1 \| \one
\end{bmatrix}&=
\left(I+\sum_{j=1}^\infty(P^j - \one \pi^T) \right) \begin{bmatrix} \|\pi_2\| \one \\
-\| \pi_1 \| \one \end{bmatrix}\\
&=
\sum_{j=1}^\infty \left( P^j \begin{bmatrix} \| \pi_2\| \one \\
-\| \pi_1 \| \one \end{bmatrix}\right)
\end{split}\]
therefore
\[
\gamma=\sum_{j=1}^\infty \left( \begin{bmatrix} \widehat \pi_1^T & -\widehat \pi_2^T 
\end{bmatrix}P^j \begin{bmatrix} \| \pi_2\| \one \\
-\| \pi_1 \| \one \end{bmatrix}\right).
\]
By partitioning $P^j$ according to the partitioning of $P$, we have
\[
P^j=\begin{bmatrix}
P_{11}^{(j)} & P_{12}^{(j)}\\
P_{21}^{(j)} & P_{22}^{(j)}
\end{bmatrix}.
\]
Since $P^j\one=\one$,
we find that 
\[
\begin{bmatrix} \widehat \pi_1^T & -\widehat \pi_2^T 
\end{bmatrix}
P^j \begin{bmatrix} \| \pi_2\| \one \\
-\| \pi_1 \| \one \end{bmatrix}=
\widehat \pi_1^T P_{11}^{(j)}\one + \widehat \pi_2^T P_{22}^{(j)}\one -1 = -\lambda_2^{(j)},
\]
where $\lambda_2^{(j)}$ is the eigenvalue different from 1 of the $2\times 2$ aggregated matrix
\[
\begin{bmatrix}
\widehat \pi_1^T P_{11}^{(j)}\one & \widehat \pi_1^T P_{12}^{(j)}\one \\
\widehat \pi_2^T P_{21}^{(j)}\one & 
\widehat \pi_2^T P_{22}^{(j)}\one
\end{bmatrix}.
\]

{
Since  $\pi^T(I-P+\one \pi^T)^{-1}=\pi^T$, $(I-P+\one \pi^T)^{-1}\one=\one$, and $\pi^T\one=1$ we may conclude that
    \[
\gamma=\left(r \pi^T+
\begin{bmatrix}
\widehat \pi_1^T & -\widehat \pi_2^T 
\end{bmatrix}\right)
(I-P+\one \pi^T)^{-1}
\left(s \one +\begin{bmatrix}
\| \pi_2\| \one \\
-\| \pi_1 \| \one
\end{bmatrix}\right)-r s,
    \]
    for any scalars $r$ and $s$.
Other expressions for $\gamma$ are
 stated by the following:
\begin{proposition}
We have
\begin{equation}\label{formula1}
\gamma = \|\pi_1\|\theta-\|\pi_2\|, 
\end{equation}
where $\theta$ has one of the following equivalent expressions
\begin{align}\label{formula for theta2}
\theta 
& = \begin{bmatrix}
\mathbf{0}^T & \hat{\pi}_2^T
\end{bmatrix}\left(I - P + \begin{bmatrix}
\one \\ \mathbf{0}
\end{bmatrix}\begin{bmatrix}
\hat{\pi}_1^T & \mathbf{0}^T
\end{bmatrix} \right)^{-1} \begin{bmatrix}
\mathbf{0} \\ \one
\end{bmatrix}\\\label{formula for theta4}
& = 
\hat{\pi}_2^T\left( I
+ (I-P_{22})^{-1}P_{21}(I-P_1+\one\hat{\pi}_1^T)^{-1}P_{12}
\right)
(I-P_{22})^{-1}\one\\\label{formula for theta3}
& = \hat{\pi}_2^T \left(I - P_{22} - P_{21}(I-P_{11}+\one\hat{\pi}_1^T)^{-1}P_{12}\right)^{-1}\one
\\\label{formula for theta5}
& = \hat{\pi}_2^T \left(I - P_2 + \frac{1}{1+\hat{\pi}_1^TS\one}P_{21}S\one
\hat{\pi}_1^TSP_{12}\right)^{-1}\one,\quad S=(I-P_{11})^{-1}.
\end{align}
\end{proposition}
\begin{proof}
From \eqref{eq:gamma}, by choosing  $\mathbf{u}^T=
\begin{bmatrix}
      \one^T & \mathbf{0}^T
\end{bmatrix}$ and $\mathbf{v}^T=\begin{bmatrix}
        \pi_1^T & \mathbf{0}^T
    \end{bmatrix} $, and by observing that 
   $\widehat Z\mathbf{u}=\one$ and $\mathbf{v}^T \widehat Z=\pi^T$, where $\widehat Z=(I-P+\mathbf{u} \mathbf{v}^T)^{-1}$,
    we arrive at \eqref{formula1}, with $\theta$ given in
    \eqref{formula for theta2}. The remaining expressions are obtained by applying formal manipulations to \eqref{formula for theta2} and properties of Schur complements.
\end{proof}
}

\section{Application of Theorem~\ref{th:kappap}}\label{sec:spec}
Here we apply the main result of the previous section to structured stochastic matrices. %

\subsection{Periodic Markov chains}

Let $P$ be the transition matrix of an ergodic Markov chain with $n$ states. The \textit{period} $p_i$ of state $i$ is the greatest common divisor of all natural numbers $m$ such that $(P^m)_{i,i}>0$. A Markov chain is called \textit{periodic} if $p_i\geq 2$ for all $1\leq i\leq n$, and it is called \textit{aperiodic} otherwise. The \textit{period} of a periodic Markov chain is the greatest common divisor of periods of all states. It is well-known that if the Markov chain is periodic, then $P$ is permutationally similar to a block-cyclic matrix (see \cite{bremaud2001markov}). 

Given a periodic Markov chain with $n$ states and period $d$, we may assume that the transition matrix $P$ is given by
\begin{equation}\label{eq:cyc}
P=\begin{bmatrix}
P_{11} & P_{12}\\
P_{21} & P_{22}
\end{bmatrix}=\left[\begin{array}{c|cccc}
0& 0 &\ldots&0&A_d\\\hline
A_1&0&\ddots&\ddots&0\\
0&A_2&\ddots&\ddots&\vdots\\
\vdots&\ddots&\ddots&0&\vdots\\
0&\ldots&0&A_{d-1}&0
\end{array}\right]
\end{equation}
where all block diagonal matrices are square and each of $A_i$'s is a rectangular stochastic matrix. Let $n_i$ be the size of $i^{\text{th}}$ block diagonal matrix of $P$ for $i=1,\dots,d$. %
A set of the states corresponding to a block diagonal matrix is called a \textit{cyclic class}. The \textit{cyclicity index} is the number of cyclic classes. One can find that each of 
\begin{equation}\label{eq:temp3}
    A_dA_{d-1}\cdots A_1,\quad A_1A_{d}\cdots A_2,\quad \dots\quad,\quad A_{d-1}\cdots A_1A_{d}
\end{equation} 
is a square stochastic matrix and its corresponding Markov chain is aperiodic. It can be seen that
\begin{equation}\label{eq:inv(I-P22)}
(I-P_{22})^{-1} = \left[\begin{array}{c|c|c|c|c}
I& 0 &                  & 0   & 0 \\
A_2& I&                 & \vdots  & \vdots \\
A_3A_2&A_3&    \cdots   &  0  & 0 \\
\vdots&\vdots&          &  I & 0 \\
A_{d-1}\cdots A_3A_2& A_{d-1}\cdots A_4A_3 &        & A_{d-1} & I
\end{array}\right].
\end{equation}
It is immediate to find that the stochastic complements $P_1$ and $P_2$ of $P_{11}$ and $P_{22}$ are, respectively,
\begin{equation}\label{eq:Pcirc}
P_1=A_dA_{d-1}\cdots A_1,\quad
P_2=\begin{bmatrix}
0&\ldots&\ldots&0&A_1A_d\\
A_2&0&\ddots&\ddots&0\\
0&A_3&\ddots&\ddots&\vdots\\
\vdots&\ddots&\ddots&0&\vdots\\
0&\ldots&0&A_{d-1}&0
\end{bmatrix}.
\end{equation}
We note that the censored Markov chain for $P_2$ is periodic.

\begin{proposition}\label{prop:3}
	Let $P$ be the transition matrix of a periodic Markov chain. Suppose that $P$ is of form \eqref{eq:cyc}. Then
	\[
	\kappa(P)=\kappa(P_1)+\kappa(P_2)+\frac12
	\]
	where $P_1$ and $P_2$ are given in \eqref{eq:Pcirc}.
\end{proposition} 
\begin{proof}
	It suffices to show that $\gamma$ in \eqref{formula0} is $\frac{1}{2}$. In order to find $\gamma$, we shall find the value of $\theta$ introduced in Equation \eqref{formula1}. From \eqref{formula for theta2}, we may look at $\theta$ as the sum of two terms $\theta_1$ and $\theta_2$, where
	\[
	\begin{split}
	& \theta_1=\hat\pi_2^T(I-P_{22})^{-1}\one,\\
	& \theta_2=\hat\pi_2^T(I-P_{22})^{-1}P_{21}(I-P_1+\one \hat\pi_1^T)^{-1}P_{12}(I-P_{22})^{-1}\one.
	\end{split}
	\]
	
	We first find expressions of $\hat\pi_2^T$ and $(I-P_{22})^{-1}\one$. Let $$\pi^T=[\bs_1^T,\bs_2^T,\ldots,\bs_d^T]$$ 
	be conformally partitioned with $P$ and be the stationary distribution vector. From the condition $\pi^TP=\pi^T$, we have 
	$\bs_{i}^T=\bs_{i+1}^TA_i$, $i=1,\ldots,d-1$, and $\bs_d^T=\bs_1^TA_d$. Then 
	$$\bs_{2}^T=\bs_{1}^TA_dA_{d-1}\cdots A_2,\quad \bs_{3}^T=\bs_{1}^TA_dA_{d-1}\cdots A_3,\quad \dots\quad,\quad \bs_{d-1}^T=\bs_{1}^TA_dA_{d-1}.$$
	Moreover, since $A_i\one=\one$, it follows that $\bs_i^T\one=\bs_{i+1}^T\one$ and $\bs_d^T\one=\bs_{1}^T\one$. Hence, $\|\bs_j\|=\frac 1d$ for $1\leq j\leq d$. Note that $\pi_1 = \bs_1$. Therefore,
	$$\pi^T=\frac1d \begin{bmatrix}
	\hat\pi_1^T & \hat\pi_1^T(A_{d}A_{d-1}\cdots A_2) & \hat\pi_1^T(A_dA_{d-1}\cdots A_3) & \cdots & \hat\pi_1^TA_{d}
	\end{bmatrix},$$
	whence 
	\begin{equation}\label{eq:pi2}
	\hat\pi_2^T=\frac1{d-1}\begin{bmatrix}
	\hat\pi_1^T(A_dA_{d-1}\cdots A_2) &
	\hat\pi_1^T(A_dA_{d-1}\cdots A_3) & \cdots & \hat\pi_1^TA_{d}
	\end{bmatrix}.
	\end{equation}
	Observing $(I-P_{22})^{-1}$ in \eqref{eq:inv(I-P22)}, we have
	\begin{equation}\label{eq:imp2}
	\begin{aligned}
	&(I-P_{22})^{-1}\one = \begin{bmatrix}
	\one^T & 2\one^T & \cdots & (d-1)\one^T
	\end{bmatrix}^T,\\ &P_{12}(I-P_{22})^{-1}\one =(d-1)\one,\\
	&(I-P_{22})^{-1}P_{21}\one=\one.
	\end{aligned}
	\end{equation}
	
	By using the expression of $\hat\pi_2$ in \eqref{eq:pi2} and the expression of $(I-P_{22})^{-1}\one$ in \eqref{eq:imp2} together with the property $A_i\one=\one$,
	we find that
	\[
	\theta_1=\frac1{d-1}\sum_{i=1}^{d-1}i\hat\pi_1^T\one=\frac12 d.
	\]
	Note that since $(I-P_1+\one\hat\pi_1^T)\one = \one$, we have $(I-P_1+\one\hat\pi_1^T)^{-1}\one = \one$. Concerning $\theta_2$, we find from \eqref{eq:imp2} that 
	\[\begin{aligned}
	\theta_2&=\hat\pi_2^T(I-P_{22})^{-1}P_{21}(I-P_1+\one \hat\pi_1^T)^{-1}P_{12}(I-P_{22})^{-1}\one=d-1.
	\end{aligned}
	\]
	Thus, $\theta=\theta_1+\theta_2=\frac32 d-1$ and therefore $\gamma = \|\pi_1\|\theta-\|\pi_2\|=\frac12$.
\end{proof}

\begin{theorem}\label{prop:above}
	Let $P$ be the transition matrix of a periodic Markov chain. Suppose that $P$ is of form \eqref{eq:cyc}. Then
	\begin{equation}\label{eq:nonsquare}
	\kappa(P) = d\kappa(A_dA_{d-1}\cdots A_1)+n-dn_1+\frac{d-1}{2}.
	\end{equation}
	Moreover, if $n_1=\cdots = n_d$ then
	\begin{equation}\label{eq:square}
	\kappa(P)=d\kappa(A_dA_{d-1}\cdots A_1)+\frac{d-1}2.
	\end{equation}
\end{theorem}
\begin{proof}
	Consider $P_2$ in \eqref{eq:Pcirc} with $d\geq 3$. Note that the Markov chain corresponding to $P_2$ is periodic. Applying Proposition~\ref{prop:3}, we find that
	$$\kappa (P_2) = \kappa(A_1A_dA_{d-1}\cdots A_2)+\kappa(P_3)+\frac{1}{2},$$
	where
	$$
	P_3 = \left[\begin{array}{c|ccc}
	0&\ldots&0&A_2A_1A_d\\\hline
	A_3&\ddots&\ddots&0\\
	&\ddots&\ddots&\vdots\\
	0& &A_{d-1}&0
	\end{array}\right].
	$$
	Hence, $\kappa(P)=\kappa(A_dA_{d-1}\cdots A_1)+\kappa(A_1A_dA_{d-1}\cdots A_2)+\kappa(P_3)+1$. If $d\geq 4$, one can apply the proposition to $P_3$ with the partition. In this manner, recursively applying the proposition, we obtain 
	\begin{equation*}
	\begin{aligned}
	\kappa(P)
	&=\kappa(A_dA_{d-1}\cdots A_1)+
	\kappa(A_1A_dA_{d-1}\cdots A_2)
	+\cdots\\
	&+
	\kappa(A_{d-1}A_{d-2}\cdots A_1A_d)+\frac{d-1}2.
	\end{aligned}
	\end{equation*}
	
	Recall that $A_i$ is of size $n_{i+1}\times n_i$ for $1\leq i\leq d-1$ and $A_d$ is of size $n_1\times n_d$. Note that $A_{i-1}A_{k-2}\cdots A_1A_dA_{d-1}\cdots A_i$ is of size $n_i\times n_i$ for $1\leq i\leq d$. From \eqref{eq:kemdiff}, we obtain 
	\begin{align*}
	\kappa(A_{k-1}A_{k-2}\cdots A_1A_dA_{d-1}\cdots A_{k})
	&=\kappa(A_{k-2}A_{k-3}\cdots A_1A_dA_{d-1}\cdots A_{k-1})\\
	&+ n_{k}-n_{k-1}
	\end{align*}
	for $2\leq k\leq d$. It follows that
	\begin{align*}
	\kappa(A_d & A_{d-1}\cdots A_1)+\cdots+\kappa(A_{d-2}\cdots A_1 A_dA_{d-1})+\kappa(A_{d-1}\cdots A_1A_d)\\
	&=\;d\kappa(A_dA_{d-1}\cdots A_1)+(d-1)(n_2-n_1)+\cdots\\
	&+\; 2(n_{d-1}-n_{d-2})+n_d-n_{d-1}
	=\;d\kappa(A_dA_{d-1}\cdots A_1)+n-dn_1.\qedhere
	\end{align*}
\end{proof}

\begin{remark}\rm
The characteristic polynomials of the matrices $P$ in \eqref{eq:cyc} and $P_1=A_dA_{d-1}\cdots A_1$ are related by the equation 
$\det (\lambda I-P)=\lambda^{\ell}\det (\lambda^d I - P_1)$, where $\ell=n - d  n_1$.
 Therefore, from the computational point of view, if the eigenvalues of $P_1$ are explicitly known, then also the eigenvalues of $P$ are explicitly known and we can recover $\kappa(P)$ 
	 from \eqref{eq:kemeig}.
	On the other hand,  if the eigenvalues of $P_1$ are not known, but the value of $\kappa(P_1)$ is available, then Kemeny's constant of $P$ can be directly obtained from \eqref{eq:nonsquare}.
\end{remark}

\begin{example}\label{cor:kemeny_bipartite}
	A collection of random walks on undirected graphs is one of the most accessible families of Markov chains. If a random walk is periodic, then the underlying graph is necessarily bipartite. Let $P$ be an irreducible stochastic matrix with the structure
	\[
	P=\begin{bmatrix}
	0&P_{12}\\P_{21}&0
	\end{bmatrix}
	\]
	where $P_{12}$ and $P_{21}$ have size $n_1\times n_2$ and $n_2\times n_1$, respectively.
	Then
	\begin{equation}\label{eq:app1}
	\kappa(P)=2\kappa(P_1)-n_1+n_2+\frac12,
	\end{equation}
	where $P_1=P_{12}P_{21}$.
\end{example}

\begin{example}
	Consider the transition matrix $P$ in \eqref{eq:cyc}. Suppose that for $i=1,\dots,d$, $A_i$ is the matrix with all entries equal to $\frac1{n_i}$. Then $A_dA_{d-1}\cdots A_1$ 
	is the $n_1\times n_1$ matrix with all entries equal to $\frac1{n_1}$. Since $\kappa(A_dA_{d-1}\cdots A_1)=n_1-1$, we have
	\begin{equation*}
	\kappa(P)= n -\frac{d+1}2.
	\end{equation*}
\end{example}

Now we provide a lower bound on Kemeny's constant of a periodic Markov chain.

\begin{corollary}\label{cor:lower bound of periodic}
	Let $P$ be the transition matrix of a periodic Markov chain. Suppose that $P=[p_{xy}]$ is of form \eqref{eq:cyc} with $n_1\leq n_i$ for $1\leq i\leq d$. Then
	\begin{equation*}
	\kappa(P) \geq n-\frac{dn_1+1}{2},
	\end{equation*}
	with equality when $p_{xy}$ is given as in \eqref{eq:p_xy}.
\end{corollary}
\begin{proof}
	From \eqref{eq:nonsquare}, $\kappa(P) = d\kappa(A_dA_{d-1}\cdots A_1)+n-dn_1+\frac{d-1}{2}$. It is known in \cite[Remark~2.14]{kirkland2010fastest} that given an $m\times m$ irreducible stochastic matrix $P$, $\kappa(P)\geq \frac{m-1}{2}$ with equality if and only if $P$ is the adjacency matrix of a directed $m$-cycle. Hence, it is enough to find a periodic Markov chain such that $A_dA_{d-1}\cdots A_1$ is the adjacency matrix of a directed $n_1$-cycle. 
	
	Note that $n_1\leq n_i$ for $1\leq i\leq d$. We may suppose that for $j = 1,\dots,d$, the cyclic class corresponding to $j^{\text{th}}$ block diagonal matrix of $P=[p_{xy}]$ is partitioned into $n_1$ subsets $C_1^j,\dots,C_{n_1}^j$. Let $p_{xy}$ be given as follows:
	\begin{align}\label{eq:p_xy}
	\begin{split}
	p_{xy} = \begin{cases}
	\frac{1}{|C_{\ell}^d|} , & \text{if $x\in C_\ell^1$ and $y\in C_\ell^d$ for $1\leq \ell\leq n_1$;}\\
	\frac{1}{|C_{\ell}^{j-1}|} , & \text{if $x\in C_\ell^j$ and $y\in C_\ell^{j-1}$ for $1\leq \ell\leq n_1$ and $3\leq j\leq d$;}\\
	1 , & \text{if $x\in C_\ell^2$ and $y\in C_{\ell-1}^1$ for $2\leq \ell\leq n_1$, or $x\in C_1^2$ and $y\in C_{n_1}^1$;}\\
	0 , & \text{otherwise.}\\
	\end{cases}
	\end{split}
	\end{align}
	It can be seen that $A_dA_{d-1}\cdots A_1$ is the adjacency matrix of a directed $n_1$-cycle, which completes the proof.
\end{proof}

%
%
%

\subsection{Kronecker product of stochastic matrices}

Given stochastic matrices $A$ and $B$, $A\otimes B$ is also stochastic where $\otimes$ denotes the Kronecker product. We will provide an expression of Kemeny's constant of $A\otimes B$ after applying Theorem~\ref{th:kappap}.

Partition $P=A\otimes B$ as follows:
\begin{equation}\label{eq:temp1}
P =\begin{bmatrix}
P_{11} & P_{12}\\
P_{21} & P_{22}
\end{bmatrix}= \left[\begin{array}{c|ccc}
a_{11}B & a_{12}B & \dots & a_{1n}B\\\hline
a_{21}B & & & \\
\vdots & & A_1 \otimes B & \\
a_{n1}B & & & 
\end{array}\right].
\end{equation}
We use $\mathbf{e}_i$ to denote the vector with a $1$ in the $i\text{th}$ entry and zeros elsewhere.
We have the following
\begin{proposition}
Let $A=[a_{ij}]$ and $B$ be stochastic matrices with stationary distribution vectors $\mathbf{x}=(x_i)$ and $\mathbf{y}$, respectively. Let $P=A\otimes B$. Then
	\[
	\kappa(P)=\kappa(P_1)+\kappa(P_2)+\frac1{1-x_1}(\mathbf{e}_1^T(I-A+\one\mathbf{x}^T)^{-1}\mathbf{e}_1-x_1),
	\]
	where $P_1$ and $P_2$ are the stochastic complements in \eqref{eq:temp1} of $P_{11}$ and $P_{22}$, respectively. 
\end{proposition} 
\begin{proof}
Let $A=[a_{ij}]$ and $B$ be matrices of size $n\times n$ and $m\times m$, respectively.
It suffices to provide an explicit expression for $\gamma$ in Theorem~\ref{th:kappap}. 

Let $\mathbf{x} = (x_i)\in\mathbb R^n$ and $\mathbf{y}\in\mathbb R^m$ be the stationary distribution vector for $A$ and $B$, respectively. Then $\mathbf{x}^TA=\mathbf{x}^T$, $\mathbf{y}^TB=\mathbf{y}^T$ and $\mathbf{x}^T\one=\mathbf{y}^T\one=1$. Let $\pi^T = \begin{bmatrix}
\pi_1^T & \pi_2^T
\end{bmatrix}$ be conformally partitioned with $P$ and be the stationary distribution vector for $P$. Since $\pi = \mathbf{x}\otimes \mathbf{y}$, we have $\pi_1 = x_1\mathbf{y}$ and $\pi_2 = (x_2,\dots, x_n)\otimes \mathbf{y}$. So, $\|\pi_1\| = x_1$ and $\|\pi_2\| = 1-x_1$.

Set $q=mn-m$. Consider
\[
\begin{bmatrix}
\|\pi_2\|\one_m\\ -\|\pi_1\|\one_q\end{bmatrix}=
-\|\pi_1\|\begin{bmatrix}
\one_m\\ \one_q
\end{bmatrix}+\begin{bmatrix}
\one_m\\\mathbf{0}_q\end{bmatrix}.
\]
Then
\begin{equation}\label{eq:temp}
(I-P+\one\pi^T)^{-1}\begin{bmatrix}
\|\pi_2\|\one_m\\ -\|\pi_1\|\one_q
\end{bmatrix}=-\|\pi_1\|\begin{bmatrix}
\one_m\\ \one_q
\end{bmatrix}+(I-P+\one\pi^T)^{-1}(\mathbf{e}_1\otimes\one_m).
\end{equation}
We claim that $$(I-P+\one\pi^T)^{-1}(\mathbf{e}_1\otimes\one_m)=\mathbf{z}\otimes\one_m$$ where $\mathbf{z}=(I-A+\one_n\mathbf{x}^T)^{-1}\mathbf{e}_1$. Consider
\[
(I-A\otimes B+\one (\mathbf{x}\otimes\mathbf{y})^T)(\mathbf{z}\otimes \one_m)=\mathbf{e}_1\otimes\one_m.
\]
This system has the form
\[
(\mathbf{z}\otimes\one_m)-(A\otimes B)(\mathbf{z}\otimes\one_m)+(\one_n\otimes\one_m)(\mathbf{x}\otimes\mathbf{y})^T(\mathbf{z}\otimes\one_m)=(\mathbf{e}_1\otimes\one_m).
\]
This implies that
\[
(\mathbf{z}-A\mathbf{z}+(\one_n\mathbf{x}^T)\mathbf{z}-\mathbf{e}_1)\otimes\one_m=\mathbf{0}.
\]
Hence, our desired claim is established. 

Now, applying \eqref{eq:temp} together with the claim to the expression of $\gamma$ given in Theorem \ref{th:kappap}, we obtain
\[\begin{split}
\gamma=&\begin{bmatrix}
\widehat \pi_1^T & -\widehat \pi_2^T 
\end{bmatrix}
(I-P+\one \pi^T)^{-1}
\begin{bmatrix}
\| \pi_2\| \one_m \\
-\| \pi_1 \| \one_q
\end{bmatrix}\\
=&\begin{bmatrix}
\hat \pi_1^T & -\hat \pi_2^T
\end{bmatrix}(-\|\pi_1\|\one+\mathbf{z}\otimes\one_m)\\
=&(\begin{bmatrix}
1 & -\frac1{1-\|\pi_1\|}(x_2,\ldots,x_n)
\end{bmatrix}\otimes\mathbf{y}^T)(\mathbf{z}\otimes\one_m)\\
=&\frac1{1-x_1}\begin{bmatrix}
1-x_1 & -x_2 & \dots & -x_n
\end{bmatrix}\mathbf{z}\\
=&\frac 1{1-x_1}(\mathbf{e}_1-\mathbf{x})^T(I-A+\one_n\mathbf{x}^T)^{-1}\mathbf{e}_1\\
=&\frac1{1-x_1}(\mathbf{e}_1^T(I-A+\one_n\mathbf{x}^T)^{-1}\mathbf{e}_1-x_1).
\end{split}\]
Therefore, we obtain the desired result.
\end{proof}

\subsection{Sub-stochastic matrices with constant row sums}

Here is the result of this subsection.

\begin{proposition}\label{prop:2}
Let \begin{equation*}
	P = \begin{bmatrix}
	P_{11} & P_{12} \\
	P_{21} & P_{22}
	\end{bmatrix},
\end{equation*} be an irreducible stochastic matrix, where $P_{11}\one = r_1\one$ and $P_{22}\one = r_2\one $ for some $0\leq r_1,r_2 < 1$. Denote by $P_1$ and $P_2$ the stochastic complements of $P_{11}$ and $P_{22}$, respectively. Then
\begin{align*}
	&\|\pi_1\| = \frac{1-r_2}{2-r_1-r_2},\quad \|\pi_2\| = \frac{1-r_1}{2-r_1-r_2},\\
	&\kappa(P) = \kappa(P_1)+\kappa(P_2)+\frac{1}{2-r_1-r_2}.
	\end{align*}
\end{proposition}
\begin{proof}
	Note that $P_{12}\one = (1-r_1)\one$ and $P_{21}\one = (1-r_2)\one$. From \eqref{formula for theta2},
	\begin{align*}
	\theta & = \hat{\pi}_2^T(I-P_{22})^{-1}\one + \hat{\pi}_2^T(I-P_{22})^{-1}P_{21}(I-P_1+\one\hat{\pi}_1^T)^{-1}P_{12}(I-P_{22})^{-1}\one\\
	& = \frac{1}{1-r_2} + \frac{1-r_1}{1-r_2}=\frac{2-r_1}{1-r_2}.
	\end{align*}
	On the other hand, we find from \eqref{eq:agg2} that the aggregated matrix $S$ is given by
	\[
	S=\begin{bmatrix}
	r_1&1-r_1\\
	1-r_2&r_2
	\end{bmatrix},
	\]
	and so  $\oldalpha_1=\|\pi_1\|= (1-r_2)/(2-r_1-r_2)$, $\oldalpha_2=\|\pi_2\|=(1-r_1)/(2-r_1-r_2)$, and $\|\pi_1\|(1+\theta)=(3-r_1-r_2)/(2-r_1-r_2)$.
	Whence, in view of \eqref{formula0} and \eqref{formula1}, we deduce that
	\[
	\kappa(P)=\kappa(P_1)+\kappa(P_2)+\frac1{2-r_1-r_2}.\qedhere
	\]
\end{proof}

\section{Some bounds}\label{sec:bounds}
Assume we are given the values of $\kappa(P_1)$ and $\kappa(P_2)$. According to Equation \eqref{formula0}, we can provide lower and upper bounds to the value of $\kappa(P)$ once we are given bounds to the constant $\gamma$.
 In order to do this, we need to determine upper and lower bounds to the value of $\|\pi_1\|$. 

From the equation $\pi^TP=\pi^T$ we find that
\[
\begin{split}
    \pi_1^T=\pi_2^T P_{21}(I-P_{11})^{-1},\quad\quad\pi_2^T =\pi_1^T P_{12}(I-P_{22})^{-1}.
\end{split}
\]
Taking the infinity norms of both sides and using the identity $\|\mathbf{x}^T\|_\infty=\|\mathbf{x}\|_1$, we obtain
\[
\|\pi_1\|\le\|\pi_2\|\, \|P_{21}\|_\infty\|(I-P_{11})^{-1}\|_\infty\le
\|\pi_2\|\frac{\|P_{21}\|_\infty}{1-\|P_{11}\|_\infty},
\]
where the latter inequality is valid if $\|P_{11}\|_\infty<1$.
A similar inequality can be obtained for $\|\pi_2\|$. Combining both inequalities, under the assumption $\|P_{11}\|_\infty,\|P_{22}\|_\infty<1$, we get
\[
\frac{1-\|P_{22}\|_\infty}{\|P_{12}\|_\infty}\le\frac{\|\pi_1\|}{\|\pi_2\|}
\le \frac{\|P_{21}\|_\infty}{1-\|P_{11}\|_\infty}.
\]
Moreover, since $\|\pi_2\|=1-\|\pi_1\|$, we obtain
\begin{equation}\label{eq:pi1bound}
\frac{1-\|P_{22}\|_\infty}{1-\|P_{22}\|_\infty+\|P_{12}\|_\infty}\le\|\pi_1\|\le 
\frac{\|P_{21}\|_\infty}{1-\|P_{11}\|_\infty+\|P_{21}\|_\infty}.
\end{equation}
Since the matrix $P$ is known, the above bounds can be actually computed  at a low computational cost.

Now a bound to the constant $\gamma$ can be obtained by relying on Equation \eqref{formula1} coupled with \eqref{formula for theta2} (one can use \eqref{formula for theta3} or \eqref{formula for theta4}). From \eqref{formula1}, $\gamma = (1+\theta)\|\pi_1\| -1$. We see from \eqref{eq:pi1bound} that, {\color{blue} if $1+\theta\ge 0$, then}
\[
(\theta+1) \frac{1-\|P_{22}\|_\infty}{1-\|P_{22}\|_\infty+\|P_{12}\|_\infty}-1
\le
\gamma
\le
(\theta+1)\frac{\|P_{21}\|_\infty}{1-\|P_{11}\|_\infty+\|P_{21}\|_\infty}-1.
\]
{\color{blue} A similar inequality holds if $1+\theta<0$.}

Note that $\|(I-P_{22})^{-1}P_{21}\|_\infty=1$. Concerning $\theta$, it follows from \eqref{formula for theta2} that
\begin{equation}\label{eq:thetabound}
     \theta\le\; \|(I-P_{22})^{-1}\|_\infty\left(1+
\|P_{12}\|_\infty
\|(I-P_1+\one\widehat\pi_1^T)^{-1}\|_\infty\right).   
\end{equation}
The upper bound to $\gamma$ is expressed in terms of $(I-P_{22})^{-1}$, $(I-P_1+\one\widehat\pi_1^T)^{-1}$ and the norms of the blocks $P_{ij}$.

For another bound on $\gamma$, consider the expression of $\gamma$ in Theorem~\ref{th:kappap}, \textit{i.e.},
\[
\gamma=\begin{bmatrix}
\widehat\pi_1^T&-\widehat\pi_2^T
\end{bmatrix}(I-P+\one\pi^T)^{-1}\begin{bmatrix}
||\pi_2\|\one\\ -\|\pi_1\|\one
\end{bmatrix}.
\]
Taking the infinity norm of both sides in the above equation yields
\[
|\gamma|\le 2\|(I-P+\one\pi^T)^{-1}\|_\infty\max(\|\pi_1\|,1-\|\pi_1\|).
\]

\subsection{A perturbation result}

Let $P$ be an $n\times n$ stochastic and irreducible matrix, and $E$ be $n\times n$ matrix such that $E\one=0$. Let $\epsilon>0$. Then $(P+\epsilon E)\one =\one$ . Assume that $P(\epsilon):=P+\epsilon E$ is stochastic and $\|E\|_\infty\le 1$.

Here, our goal is to relate $\kappa(P(\epsilon))$ and $\kappa(P)$. We have
\[\begin{split}
\kappa(P(\epsilon))=\hbox{Tr}( (I-P(\epsilon)+\one \mathbf{h}^T)^{-1}),\quad\quad\kappa(P)=\hbox{Tr}( (I-P+ \one \mathbf{h}^T)^{-1}),
\end{split}
\]
where $\mathbf{h}$ is any vector such that $\mathbf{h}^T\one=1$.
By subtracting both sides of the above equations we get
\[
\kappa(P(\epsilon))-\kappa(P)=
\hbox{Tr}((I-P(\epsilon)+ \one \mathbf{h}^T)^{-1}(P(\epsilon)-P)(I-P+\one \mathbf{h}^T)^{-1}).
\]
Neglecting $O(\epsilon^2)$ terms, we obtain
\[\begin{split}
\kappa(P(\epsilon))-\kappa(P)&=
\epsilon\, \hbox{Tr}((I-P+\one \mathbf{h}^T)^{-1}E(I-P+\one \mathbf{h}^T)^{-1})\\
&=\epsilon\,\hbox{Tr}((I-P+ \one \mathbf{h}^T)^{-2}E)+O(\epsilon^2).
\end{split}
\]
This estimate can be also deduced as a specific case of \cite[Lemma 3.2]{kirk}.

From the Cauchy--Schwarz inequality, we have $\hbox{Tr}(AB)\le \|A\|_F\|B\|_F$, where $\|\cdot\|_F$ is the Frobenius norm. It follows that
\begin{equation}\label{eq:perturb}
|\kappa(P(\epsilon))-\kappa(P)|\le
\epsilon \|(I-P+ \one \mathbf{h}^T)^{-2}\|_F\|E\|_F+O(\epsilon^2).
\end{equation}

Combining the above bound with Example~\ref{cor:kemeny_bipartite} yields the following result that concerns Kemeny's constant of a matrix associated with an almost bipartite graph.

\begin{corollary}
Let $P$ be the stochastic matrix defined in Example \ref{cor:kemeny_bipartite} and $E$ be a matrix such that $E\one=0$ and $\|E\|_\infty\le 1$. Set $P(\epsilon)=P+\epsilon E$ and assume that $P(\epsilon)$ is stochastic in a neighborhood of 0. Then, up to within $O(\epsilon^2)$ terms we have
\[
\left|\kappa(P(\epsilon))-\kappa(P_{1})-\kappa(P_{2})-\frac12\right|\le \epsilon \|(I-P+ \one \mathbf{h}^T)^{-2}\|_F\|E\|_F+O(\epsilon^2),
\]
where $\mathbf{h}$ is any vector such that $\mathbf{h}^T\one =1$
\end{corollary}

Similar bounds can be obtained for the ``stochastic'' perturbation of a cyclic matrix having arbitrary cyclicity index and for the perturbation of the Kronecker product of two stochastic matrices.

Analogously, we may provide a ``perturbed version'' of Proposition \ref{prop:2} obtained by combining Proposition \ref{prop:2} itself with Equation \eqref{eq:perturb}.

\section{A divide-and-conquer algorithm}\label{sec:algo}

We can use the analysis done in Section~\ref{sec:stochastic-complements} to construct a divide-and-conquer algorithm for computing Kemeny's constant of a stochastic sparse matrix $P$. In fact, it is sufficient to identify a partitioning of the form~\eqref{eq:P} to start a recursion procedure. This procedure employs Theorem~\ref{th:kappap} to express $\kappa(P)$ in terms of $\kappa(P_1)$ and $\kappa(P_2)$ of the censored chains and the $\gamma$, where $P_1$ and $P_2$ are the stochastic complements. Algorithm~\ref{alg:kemenyandconquer} compactly describes the entire procedure in a recursive formulation.

\begin{algorithm}[htbp]
\SetKwProg{Fn}{}{}{}\SetKwFunction{FRecurs}{void FnRecursive}%
\SetKwFunction{FRecurs}{Kemeny}%
\Fn(\tcc*[h]{Implementation as a recursive function}){\FRecurs{$P$,$\pi$}}{
\KwIn{Stochastic matrix $P$, stationary distribution vector $\pi$, an integer $n_0>0$.}
\KwOut{Kemeny's constant $\kappa(P)$ of $P$.}
\tcc{In the following the $\gets$ means variable assignment while the $=$ are used to define shorthand.}
$n \gets \operatorname{size}(P)$\;
\eIf{$n < n_0$}{%
    $\kappa(P) \gets \Tr\left( (I - P + \nicefrac{\mathbf{1}\mathbf{1}^T}{n} )^{-1}\right) - 1$\;
    \Return
}{%
    $m \gets \lfloor \nicefrac{n}{2} \rfloor$\;
    $P_{11} = P(1:m,1:m)$\tcc*[r]{Block structure}
    $P_{12} = P(1:m,m+1:n)$\;
    $P_{21} = P(m+1:n,1:m)$\;
    $P_{22} = P(m+1:n,m+1:n)$\;
    $\pi_1=\pi(1:m)$\;
    $\pi_2=\pi(m+1:n)$\;
    $\hat \pi_i \gets \frac 1{\|\pi_i \|}\pi_i,~~i=1,2$\;
    \tcc{Remark: we can compute the $LU$ factorization of $I-P_{22}$ once and reuse it.}
    $ P_1 \gets P_{11}+P_{12}(I-P_{22})^{-1}P_{21}$\label{alg:system1}\tcc*[r]{stochastic complements}
    $ P_2 \gets P_{22}+P_{21}(I-P_{11})^{-1}P_{12}$\label{alg:system2}\;
    $\mathbf{x} \gets (I - P_{22})^{-1}\one$\label{alg:system3}\tcc*{Computation of $\theta$ using~\eqref{formula for theta2} }
    $\mathbf{y} \gets (I-P_1+\one\hat{\pi}_1^T)^{-1} (P_{12} \mathbf{x})$\label{alg:system4}\;
    $\mathbf{y} \gets (I-P_{22})^{-1}(P_{21}\mathbf{y})$\label{alg:system5}\;
    $\theta \gets \hat{\pi}_2^T(\mathbf{x} + \mathbf{y})$\;
    $\gamma \gets \|\pi_1\|\theta-\|\pi_2\|$\tcc*{Compute correction as in~\eqref{formula1}}
    $\kappa(P_1) \gets \;$\FRecurs{$P_1$,$\hat \pi_1$}\tcc*{Recursion}
    $\kappa(P_2) \gets \;$\FRecurs{$P_2$,$\hat \pi_2$}\tcc*{Recursion}
    $\kappa(P) \gets \kappa(P_1)+ \kappa(P_2) + \gamma$\;
    \Return
}%
}
\caption{Divide-and-conquer algorithm for the computation of Kemeny's constant.}\label{alg:kemenyandconquer}
\end{algorithm}
To effectively apply this strategy, it is crucial to efficiently address several computational subproblems. Let us set aside the computation of the stationary distribution vector $\pi$ for the starting chain $P$, which is the essential component required to initiate the entire procedure (to this regard, for large-scale problems, the algorithms proposed in \cite{aggregationdis} might be used). The most significant part of the computation lies in the solution of the linear systems to lines \ref{alg:system1}--\ref{alg:system5} of Algorithm~\ref{alg:kemenyandconquer}. 

Let us focus on the solution of the following systems:
\[
(I - P_{ii}) \mathbf{x} = \mathbf{b}, \qquad i = 1,2,
\]
where we further assume that the block matrices $\{P_{ii}\}_{i=1}^{2}$ can benefit from sparse storage. By construction, the blocks $P_{ii}$, $i=1,2$ are sub-stochastic, i.e., $(P_{ii})_{p,q} \geq 0$ for all $p,q$, and $P_{ii}\one \le \one$, $P_{ii}\one \neq \one$. Since $P$ is irreducible, the matrices  $I - P_{ii}$, $i=1,2$, are non-singular M-matrices. This property allows us to resort to different efficient iterative strategies for the solution of the systems involved. Since in general matrices will be non-symmetric, one can consider using GMRES~\cite[Section~6.5]{SaadBook} or BiCGstab~\cite[Section~7.4.2]{SaadBook} for solving the different linear systems. (In any case it will be necessary to have a preconditioner available to accelerate the convergence of the Krylov method in question.) Since we are working with M-matrices, a natural choice is to use incomplete factorizations. Specifically, we can use Incomplete LU factorizations (ILU), that is, we can approximate the matrices as
\[
    I - P_{ii} = \widetilde{L}_{i}\widetilde{U}_{i} + R_i \approx \widetilde{L}_{i}\widetilde{U}_{i} \qquad i=1,2, 
\]
with the residual matrix $R_i$ satisfying certain constraints, such as having zero entries in some prescribed locations--either static, determined on the base of the natural occurring fill-in during the computation, or via thresholding on their entries. For M-matrices, the existence of such objects is guaranteed by the fact that Gaussian elimination and non-diagonal dropping of the entries preserves the property of being a non singular M-matrix--see~\cite{MR0117242} for the original proof or~\cite[Theorem~10.1]{SaadBook} for a modern explanation. Similarly, to precondition the system on the line~\ref{alg:system4} of the algorithm that contains the matrix $P_1$, we can consider an incomplete factorization of an approximation of the matrix $I-P_1$, namely,
\[
\begin{split}
    I-P_1 = &\; I -P_{11} -P_{12}(I-P_{22})^{-1}P_{21} \\ \approx&\;  I -P_{11} -P_{12}(\operatorname{diag}(I-P_{22}))^{-1}P_{21} \approx \widetilde{L}_3\widetilde{U}_3,
\end{split}
\]
where $\operatorname{diag}(I-P_{22})$ is the diagonal matrix formed by the diagonal entries of $I-P_{22}$.
It is known that the computation of LU factorization, incomplete or not, benefits from the reordering of the entries of the matrix itself, see, e.g.,~\cite{MR1694677}. %
Furthermore, as seen in Proposition~\ref{prop:2}, the closer a matrix is to block diagonal after appropriate permutation, the easier the global calculation will be. Since the off-diagonal matrices in the block decomposition~\eqref{eq:P} will be made up of a few nonzero elements, the natural choice for the permutation algorithm is to use Nested Dissection permutation algorithm~\cite{MR1639073}; an example of the application of this algorithm is given in Figure~\ref{fig:nested_dissection}, from which we observe that the resulting matrix has the desired ``quasi block-diagonal'' structure.
\begin{figure}[htbp]
    \centering
    \includegraphics[width=0.8\columnwidth]{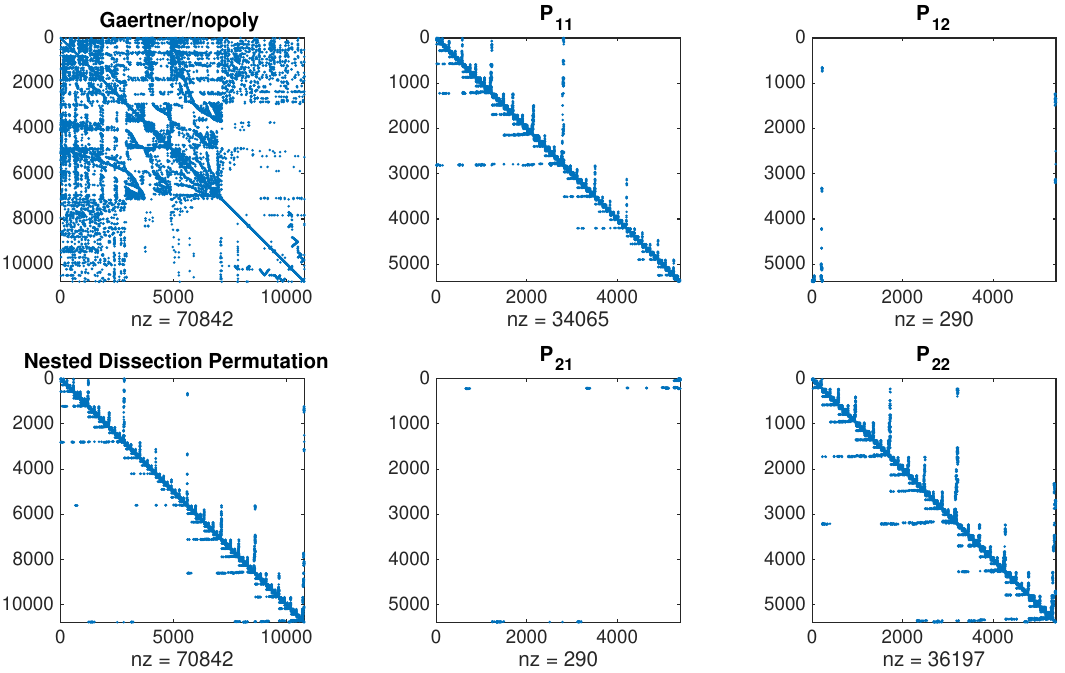}
    \caption{Application of the Nested Dissection algorithm from~\cite{MR1639073} to the \texttt{Gaertner/nopoly} matrix from the \texttt{SuiteSparse Matrix Collection} (formerly the University of Florida Sparse Matrix Collection)~\cite{MR2865011}.}
    \label{fig:nested_dissection}
\end{figure}
We want to underline that this initial permutation step would also be advisable if one wants to obtain Kemeny's constant directly from the expression~\eqref{eq:ks}, since the related effect of this choice is also that of a permutation which is fill-reducing for the computation of $LU$ factorization of sparse matrices, a step needed for the efficient computation of the matrix inverse in~\eqref{eq:ks}.

\subsection{Low-precision randomized approximation}\label{sec:low-precision}
We note that since we want to solve intermediate linear systems using an iterative method, what we are actually calculating is an approximation of Kemeny's constant. For this reason, it makes sense to consider randomized algorithms for the direct approximation of~\eqref{eq:ks} through the trace of a matrix.  Such  approaches  have been also used for instance in \cite{liMC}, for reversible Markov chains, and in \cite{Xukem}, for Markov chains modeling a random walk on an undirected graphs. We consider here the application of the \hutchpp\ algorithm~\cite{MR4537958}, which is a straightforward improvement on the Hutchinson estimator~\cite{MR1075456} used in~\cite{Xukem}.  In our case this means having an \emph{oracle} that computes
\begin{equation}\label{eq:oracle}
     \mathbf{y} = (I-P+\one \mathbf{h}^T)^{-1} \mathbf{x} \equiv A \mathbf{x}, \quad \mathbf{h} = \mathbf{1}/n.
\end{equation}
Then the following results gives us the number of oracle calls, i.e., linear system solutions, we have  to approximate $\Tr\left((I-P+\one \mathbf{h}^T)^{-1} \right)$ within a given tolerance.
\begin{theorem}[{\cite[Theorem~1]{MR4537958}}]\label{thm:hutchpp}
If \hutchpp\ is implemented with $\ell = O( \nicefrac{\sqrt{\log(\nicefrac{1}{\delta})}}{\epsilon} + \log(\nicefrac{1}{\delta}))$ matrix-vector multiplication queries, then for any positive semidefinite matrix $A$, with probability $\geq 1 - \delta$, the output of $\hutchpp(A)$ satisfies
\[
(1-\varepsilon) \Tr(A) \leq \hutchpp(A) \leq (1+\varepsilon) \Tr(A).
\]
\end{theorem}
Similarly to what was done for the solution of linear systems in the divide-and-conquer algorithm, we can use the PCG with an Incomplete Cholesky preconditioner calculated on $I-P$, i.e.,
\[
 A \approx \widetilde{L}^{-T} \left( I - \frac{1}{1 + (\mathbf{h}^T\widetilde{L}^{-T})(\widetilde{L}^{-1}\one) }(\widetilde{L}^{-1}\one)(\mathbf{h}^T\widetilde{L}^{-T}) \right) \widetilde{L}^{-1},
\]
as an oracle for the calculation of the products $A\mathbf{x}$ necessary for \hutchpp.
We point out that the 
\hutchpp\ algorithm works under the assumption that the matrix of which we approximate the trace is symmetric positive semidefinite, and this assumption is verified for Markov chains modeling a random walk on an undirected graph.

\subsection{Numerical examples}

This section contains some numerical examples in which the performance of the algorithms obtained starting from the theoretical analysis is analyzed. All experiments are reproducible starting from the code contained in the repository \href{https://github.com/Cirdans-Home/Kemeny-and-Conquer}{github.com/Cirdans-Home/Kemeny-and-Conquer}. All experiments are performed on a vertex of the Toeplitz cluster at the University of Pisa equipped with an Intel\textsuperscript{\textregistered} Xeon\textsuperscript{\textregistered} CPU E5-2650 v4 at 2.20GHz and 250 Gb of RAM, using MATLAB 9.10.0.1602886 (R2021a). 
The Markov chains used in the examples are built employing matrices from the \texttt{SuiteSparse Matrix Collection} (formerly the University of Florida Sparse Matrix Collection)~\cite{MR2865011}. Specifically, we build the probability transition matrix $P$ from irreducible adjacency matrices $A$ as
\[
    P = \operatorname{diag}(\hat{A} \one)^{-1} \hat{A},
\]
for $\hat{A}$ the matrix obtained from $A$ with all weights set to $1$. For the low-precision randomized case we use instead
\[
P = \operatorname{diag}( \hat{A} \one)^{-\nicefrac{1}{2} } \hat{A}  \operatorname{diag}( \hat{A} \one)^{-\nicefrac{1}{2} }, 
\]
for cases with $A = A^T$, i.e., the graph is undirected. All the relative errors in the following experiments are computed with respect to Kemeny's constant obtained directly, i.e., applying~\eqref{eq:ks} by computing the whole matrix inverse.

\subsubsection{Low-precision randomized approximation}

We first focus on using randomized estimators from Section~\ref{sec:low-precision} for the trace of a matrix. Table~\ref{tab:randomized} contains the estimates obtained using \hutchpp\ with the parameters $(\delta,\epsilon)$ of Theorem~\ref{thm:hutchpp} chosen as $\delta = \nicefrac{1}{4}$, and $\epsilon = 10^{-1}$, i.e., we use a set of $l = 13$ random sample vectors.
\begin{table}[htbp]
    \centering
    \begin{tabular}{lcccc}
    \toprule
    & & \multicolumn{2}{c}{Time (s)} \\
    \cmidrule{3-4}
    Matrix & $n$ & Direct & \hutchpp & Rel. Error \\
    \midrule
\texttt{Pajek/USpowerGrid} & 4941 & 2.11 & 0.43 & 2.69e-04\\
\texttt{Gaertner/nopoly} & 10774 & 21.97 & 2.64 & 7.58e-03\\
\texttt{Gleich/minnesota} & 2640 & 0.33 & 0.15 & 1.94e-03\\
    \bottomrule
    \end{tabular}
    \caption{Randomized approximation of Kemeny's constant with \hutchpp algorithm. The relative error and the timings (s) for the \hutchpp is the average over 100 repetitions. The inner iterative solver is PCG preconditioned by the ICHOL(0)-based preconditioner from~\ref{sec:low-precision} and with a tolerance of $10^{-2}$. The number of samples is $l=13$, that corresponds to~$\delta = 1/4$ and~$\epsilon = 10^{-1}$.}
    \label{tab:randomized}
\end{table}
As an internal solver for the oracle calculation we use the PCG preconditioned with an ICHOL(0), i.e., such as to preserve the sparsity pattern of the starting matrix in $\widetilde{L}$. Since the external precision we expect to achieve is of the order of $10^{-2}$, the tolerance for the iterative method is chosen to be $10^{-3}$. In all cases, we always start from the matrix reordered with the Nested Dissection permutation algorithm~\cite{MR1639073}.

\subsubsection{Divide-and-conquer algorithm}
Here we test the application of the divide-and-conquer strategy on the transition matrices built from the matrices in \texttt{SuiteSparse} \texttt{Matrix} \texttt{Collection}~\cite{MR2865011} in the previous section. We use two variants of Algorithm~\ref{alg:kemenyandconquer}: one called Recursive in Table~\ref{tab:num_example} exploits an incomplete ILU(0) factorization and the GMRES to solve systems with left term $I-P_{22}$, and the other called direct-recursive exploits the LU factorization of the matrix, already needed for the calculation of the stochastic complement, also for the solution of the two auxiliary systems with the same matrix for the calculation of $\theta$. %
\begin{table}[hbtp]
    \addtolength{\tabcolsep}{-3pt} 
    \centering
    \scriptsize
    \begin{tabular}{lccccccc}
     \toprule
     & & & \multicolumn{3}{c}{Time (s)} & \multicolumn{2}{c}{Rel. Error }\\
     \cmidrule{4-6}\cmidrule{6-8}
     Matrix & $n$ & $\kappa(P)$ & Direct & Recursive & Dir-Rec & Recursive & Dir-Rec \\
     \midrule
        \texttt{Gaertner/big} & 13209 & 58134.53 & 17 & 5.81 & \textbf{4.15} & 1.57e-08 & 4.09e-09 \\
        \texttt{vanHeukelum/cage10} & 11397 & 15378.12 & \textbf{11.06} & 44.83 & 56.76 & 8.28e-13 & 1.25e-11 \\
        \texttt{vanHeukelum/cage11} & 39082 & 51177.08 & \textbf{315.27} & 1259.11 & 1591.04 & 2.59e-10 & 1.53e-10 \\
        \texttt{HB/gre\_1107} & 1107 & 1483.57 & \textbf{0.04} & 0.10 & 0.11 & 1.97e-09 & 1.09e-09 \\
        \texttt{Gaertner/nopoly} & 10774 &  171656.87 & 9.56 & 9.14 & \textbf{6.71} & 3.25e-07 & 3.09e-07 \\
        \texttt{Gaertner/pesa} & 11738 & 131250.78 & 11.45 & 6.36 & \textbf{3.30} & 2.15e-07 & 2.28e-07 \\
        \texttt{Gleich/usroads-48} & 126146 & 1818057.53 & 8243.57 & 743.37 & \textbf{542.88} & 8.65e-07 & 1.60e-07 \\
        \texttt{Barabasi/NotreDame\_www}\textsuperscript{$\dagger$} & 34643 & 1173610.94 & 172.36 & 94.28 & \textbf{93.20} & 2.85e-04 & 3.58e-05 \\ 
        \texttt{Pajek/USpowerGrid} & 4941 & 30166.55 & 1.32 & 1.48 & \textbf{0.58} & 3.01e-08 & 2.80e-08 \\
        \texttt{Gleich/minnesota}\textsuperscript{$\dagger$} & 2640 & 18243.53 & 0.30 & 0.34 & \textbf{0.22} & 8.37e-08 & 3.53e-08  \\
    \bottomrule
    \end{tabular}
    \addtolength{\tabcolsep}{3pt}  
\caption{Performance of the recursive implementation of the divide-and-conquer algorithm for computing Kemeny's constant on some test matrices. If the matrix name has a $\dagger$ symbol, then the experiment has been run on the largest connected component of the graph, i.e, on the largest irreducible sub-chain.}  \label{tab:num_example}
\end{table}
From the results in Table~\ref{tab:num_example} we observe that in most cases the divide-and-conquer algorithm manages to reduce the computational time compared to the direct computation of Kemeny's constant. In some cases we observe the absence of an improvement, investigating in detail what we observe is that the decomposition into blocks done by halving is far from being optimal. This causes both the creation of denser stochastic complements and higher solution times for auxiliary linear systems. The implementation of nested dissection in MATLAB does not have in output the limitation of the \emph{clusters} obtained, being able to use that should increase the advantage of the recursive version compared to the one in which the recursion is done by simple halving.

\section{Conclusions}\label{sec:conc}
Kemeny's constant $\kappa(P)$ of a stochastic matrix $P$ has been expressed in terms of Kemeny's constants of the stochastic complements obtained from a block partitioning, and the constant $\gamma$. 
Explicit expressions of $\kappa(P)$ have been provided for the transition matrix of a periodic Markov chain, the Kronecker product of stochastic matrices, and sub-stochastic matrices with constant row sums. %
The main result, Theorem~\ref{th:kappap}, has been used to design a divide-and-conquer algorithm for recursively computing $\kappa(P)$. Numerical experiments applied to real-world graphs show the effectiveness of this approach especially in the case of nearly completely decomposable matrices.

As Kemeny's constant measures the expected time of a Markov chain $X$ to travel between two randomly chosen states, a natural question arises: ``What interpretation can be ascribed to Kemeny's constant of a censored Markov chain $X_1$ associated with the original Markov chain $X$?''. Since $X_1$ is induced from $X$, it would be interesting to provide insights into what features of $X$ are captured in Kemeny's constant of $X_1$. Partitioning the state space of $X$ into two subsets yields two censored Markov chains. We can see from Theorem~\ref{th:kappap} that Kemeny's constants of these censored Markov chains are interdependent. Understanding how partitioning the state space influences Kemeny's constants of censored Markov chains would be beneficial for gaining insights into the question.

\section*{Acknowledgments}
The authors sincerely thank the anonymous referee for the thorough review and valuable suggestions, which have significantly enhanced the clarity and quality of the presentation.

\bibliographystyle{abbrv}

\bibliography{KemenyCensored.bib}
\end{document}